\documentclass{amsart}%[12pt]{amsart}%{article}%[leqno]{article}
\usepackage[T1]{fontenc}
\usepackage[cp1250]{inputenc}
\usepackage{amsthm}
\usepackage{amsmath}
\usepackage{amssymb}
\usepackage{amsfonts}
\usepackage{amscd}
\usepackage{enumerate}
\usepackage{graphicx}
\usepackage{multicol}
\usepackage{mathrsfs}
\makeindex
\date{\today}

\newcommand{\M}{\mathbb{M}}

\renewcommand{\le}{\leqslant}
\renewcommand{\ge}{\geqslant}

\newcommand{\LL}{\mathbb{L}}
\newcommand{\cal}{\mathcal}

\newcommand{\PP}{\mathbb{P}}

\newcommand{\cc}{\mathbf{\mathbb{C}}}
\newcommand{\nn}{\mathbf{\mathbb{N}}}

\newcommand{\ord}{\operatorname{ord}}
\newcommand{\graph}{\operatorname{graph}}
\newcommand{\OO}{\operatorname{\mathcal{O}}}

\newcommand{\II}{\operatorname{\mathcal{I}}}
\newcommand{\JJ}{\operatorname{\mathcal{J}}}

\newcommand{\rk}{\operatorname{rank}}

\newtheorem{twB}{\textsc{Bertini theorem}}%[section]
\newtheorem{tw}{\textsc{Theorem}}[section]

\newtheorem{wn}[tw]{\textsc{Corollary}}
\newtheorem{lem}[tw]{\textsc{Lemma}}
\newtheorem{pr}[tw]{\textsc{Proposition}}
\newtheorem{fact}[tw]{\textsc{Fact}}
\theoremstyle{remark}
\newtheorem{exa}[tw]{\textsc{Example}}
\theoremstyle{remark}
\newtheorem{rem}[tw]{\textsc{Remark}}

\begin{document}
\baselineskip15pt

\title[Effective Bertini theorem]{Effective Bertini theorem\\
 and formulas for multiplicity\\ and the local {\L}ojasiewicz exponent}
\author{Tomasz Rodak, Adam R\'o\.zycki, Stanis{\l}aw Spodzieja}

\address[Tomasz Rodak]{Faculty of Mathematics and Computer Science,
University of Lodz, S. Banacha 22,
90-238 {\L}\'od\'z, Poland}
\email{tomasz.rodak@wmii.uni.lodz.pl}

\address[Adam R\'o\.zycki]{Faculty of Mathematics and Computer Science,
University of Lodz, S. Banacha 22,
90-238 {\L}\'od\'z, Poland.}
\email{adam.rozycki@wmii.uni.lodz.pl}

\address[Stanis{\l}aw Spodzieja]{Faculty of Mathematics and Computer Science,
University of Lodz,
S. Banacha 22,
90-238 {\L}\'od\'z, Poland}
\email{stanislaw.spodzieja@wmii.uni.lodz.pl}

%\thanks{This research was partially supported by the Polish National Science Centre, grant 2012/07/B/ST1/03293.}
%This research was partially supported by the Polish OPUS Grant No 2012/07/B/ST1/03293.}

\begin{abstract}
The classical Bertini theorem on generic intersection of an algebraic set with hyperplanes states the following:
\emph{Let X be a nonsingular closed subvariety
of $\mathbb{P}^n_k$, where $k$ is an algebraically closed field. Then there exists a hyperplane $H\subset \mathbb{P}^n_k$ not containing $X$ and such that the scheme $H\cap X$ is regular at every point.  Furthermore, the set of hyperplanes with this property forms an
open dense subset of the complete linear system $|H|$ considered as a projective space. }
We will show that one can effectively indicate a finite family of hyperplanes $ H $ such that at least one of them satisfies the assertion of the Bertini theorem.
As an application of the method used in the proof we will give effective formulas for the multiplicity and the {\L}ojasiewicz exponent of polynomial mappings.
\end{abstract}
\maketitle

\section*{Introduction}

The classical Bertini theorem on generic intersection of an algebraic set with hyperplanes is as follows (see   \cite[Theorem II.8.18]{H}):

\begin{twB}%[Bertini's Theorem]
\label{Bertiniclasical}
Let X be a nonsingular closed subvariety
of $\mathbb{P}^m_k$, where $k$ is an algebraically closed field. Then there exists a hyperplane $H\subset \mathbb{P}^m_k$ not containing $X$ and such that the scheme $H\cap X$ is regular at every point. (In fact, %we will see later (III, 7.9.1) that
if $\dim X \ge 2$, then $H \cap X$ is connected, hence irreducible, and so $H\cap X$ is a nonsingular variety.) Furthermore, the set of hyperplanes with this property forms an
open dense subset of the complete linear system $|H|$ considered as a projective space.
\end{twB}

This theorem is one of the fundamental tools in algebraic geometry. For more details we refer to  \cite{Kleiman}.

We will show that one can effectively indicate a finite family of hyperplanes $H$ such that  at least one of them satisfies the assertion of the Bertini theorem (see Theorems \ref{Bertini}, \ref{BertiniProj} and \ref{Bertinimultiple}). Theorems \ref{Bertini} and \ref{Bertinimultiple} are formulated in terms of algebraic cones, i.e., algebraic sets defined by homogeneous polynomials. The most general is Theorem \ref{Bertinimultiple} which says (in the projective case)  the following:
\emph{Let $d,m,q,s$ be positive integers such that $m\ge q \ge s$, let
$$
\ell=  d^{m-q}[(m-q)(d-1)+q]+mq-1,
$$
 and let $N_j:\cc^{(m+1)q}\to\cc$, $1\le j\le \ell$, be a system of independent linear functions, i.e., any $(m+1)q$ of these functions are linearly independent. Then
\begin{multline*}
\!\!\! K_{s,j_1,\ldots,j_{(m+1)q-1}}=\{x\in\PP^m:\exists_{a=(a_1,\ldots,a_{q})\in (\PP^m)^{q}}\;N_{j_1}(a)=\cdots =N_{j_{(m+1)q-1}}(a)=0,\\  a_1x^T=\cdots=a_{s}x^T=0\}\quad\hbox{for } 1\le j_1<\cdots<j_{(m+1)q-1}\le \ell
\end{multline*}
is a system of linear subspaces of dimension $m-s$ such that for any irreducible and smooth algebraic set $V\subset \PP^m$ of dimension $q$ and degree  $d$, there are $1\le j_1<\cdots<j_{(m+1)q-1}\le \ell$ such that
 the intersection $X=V\cap K_{s,j_1,\ldots,j_{(m+1)q-1}}$ is transversal at any point $x\in X$ and the set  $X$ is smooth. If moreover $q \ge s+1$ then $X$ is irreducible.}

As a consequence of the Bertini theorem we obtain Corollary \ref{Bertinimultipleextra} which says that
\emph{for any  algebraic set $V=V(f_1,\ldots,f_r)\subset \cc^m$ of pure dimension $ q$, where $f_j\in\cc[x]$, $\deg f_j\le d$ for $1\le j\le r$, we have
$$
\deg _0V=\max_{1\le j_1<\cdots<j_{m(q-1)-1}\le \ell}\deg_0 (V\cap K_{s,j_1,\ldots,j_{m(q-1)-1}}) \quad\hbox{for any $1\le s\le q$.}
$$
}

A crucial role in the proof of the Bertini theorem is played by the following observation (see Lemma \ref{arange0}): \emph{Let $m,q,d$ be positive integers with $m\ge 2$, $m\ge q$, let $\ell=d(m-q)+q$, and let $N_j:\cc^m\to\cc$, $1\le j\le \ell$, be a system of independent linear functions. Then for any algebraic cone $C_0\subset \cc^m$ with $\dim C_0\le q$ and with total degree $\delta(C_0)\le d$, there are $1\le i_1<\cdots<i_q\le \ell$ such that}
\begin{equation*}%\label{arrangeN}
C_0\cap V(N_{i_1},\ldots,N_{i_q})=\{0\}.% \quad\hbox{for some }i\in\{1,\ldots,s\}.
\end{equation*}
The number $\ell$ in the above observation is optimal in terms of $d$, $m$ and $q$ (see Remark \ref{lemthebest}).

There are a lot of numerical invariants that could be associated with a polynomial map $f$.
In this note we are interested in two of them, namely the multiplicity and the Łojasiewicz exponent.
The multiplicity is a basic biholomorphic invariant characterizing the singularity of $f$ at zero (see \cite{ATW,Fulton,Lo3,M,Serre}). As a corollary from the above observation we will prove the following effective formula for the multiplicity $i_0(f)$ of a finite polynomial mapping $f $ (see Theorem \ref{twmultiplicityformula}):
%\begin{tw}\label{twmultiplicityformula}
\emph{Let $\ell=  d^{n}(m-n)+n$.
For any independent system $L_j:\cc^m\to\cc$, $1\le j\le \ell$,
and any polynomial mapping $f\colon\cc^n\to\cc^m$ of degree $d$,  finite at $0$, we have}
$$
i_0(f)=\min_{1\le i_1<\cdots<i_n\le\ell}% d^{n-1}(m-n)+n}
\dim_\cc\OO/(L_{i_1}\circ f,\ldots,L_{i_n}\circ f).
$$
%\end{tw}
Note that in the case $m>n$, we only have $i_0(f)\ge \dim_\cc\OO/(f_1,\ldots,f_m)$.

For the {\L}ojasiewicz exponent (see \cite{Lo1,Lo2,KS1,KSS}, the definition is given in Section~\ref{Lojexpsect}),  we obtain the following formulas (see Corollaries \ref{reduct1} and \ref{reduct11}):
\emph{Let $m\ge n$ be positive integers, let $\ell=d^n(m-n)+n$, and let $L_j:\cc^m\to\cc$, $1\le j\le \ell$, be a system of independent linear functions. Then for any polynomial mapping $f:(\cc^n,0)\to(\cc^m,0)$ finite at $0$ with $\deg f\le d$ we have
$$
{\cal{L}}_0(f)=\min_{1\le i_1<\cdots<i_n\le \ell}{\cal{L}}_0((L_{i_1},\ldots,L_{i_n})\circ f)=\min_{1\le i_1<\cdots<i_n\le \ell}{\cal{L}}_0(H_{f,(L_{i_1},\ldots,L_{i_n})}),
$$
where $
H_{f,L}(z)=L\circ f(z)+(z_1^{{d^n}+1},\ldots,z_n^{{d^n}+1})$, $z=(z_1,\ldots,z_n)\in\cc^n$
for a linear mapping $L:\cc^m\to\cc^n$.}
Then by using  P{\l}oski's formula for the {\L}ojasiewicz exponent of mappings $\cc^n\to\cc^n$ (see \cite{P2,RS}) we obtain an effective formula for the exponent in terms of orders of polynomials $P_{f,L,N}\in\cc[y_1,\ldots,y_n,t]$ defining the sets of values of the mappings
$$%\begin{equation}\label{1}
\Phi_{f,L,N}(z)=(H_{f,L}(z),N(z)), %\qquad N\in \LL(n,a).
$$
where $N:\cc^n\to\cc$ is a linear function.
More precisely, we obtain the following formula (see Theorem \ref{maintheorem1}): \emph{Let  $\ell_L=d^n(m-n)+n$, $\ell_N=n+[(d^n+1)^n-1]n(n-1)$, and let
$\mathbb{I}_{L}=\{{\bf s}=(s_1,\ldots,s_n)\in \nn^n:  1\le s_1<\cdots<s_n\le \ell_L\}$. Then for any sequences $N_i:\cc^n\to\cc$, $1\le i\le \ell_N$, $L_j:\cc^m\to\cc$, $1\le j\le \ell_L$, of independent linear functions we have
$$
{\cal{L}}_0(f)=\min_{{\bf s}\in \mathbb{I}_L}\max_{1\le i\le\ell_N}\frac{1}{\Delta (P_{f,L_{\bf s},N_i})},
$$
where
$$%\begin{equation}
  \Delta(P)=\min_{j=0}^r\frac{\ord_y P_j}{r+1-j}
  %\label{e101}
$$
if $P(y,t)=\sum_{j=0}^\infty P_j(y)t^j\in \cc\{y,t\}$ is a  power
series regular of order $r+1$ in~$t$.}

%\begin{rem}\label{rem2tomainthm1}
In \cite[Theorem 7]{RS} we obtained a similar result to Theorem \ref{maintheorem1}. There, however, we used the linear mappings $L$ and $N$ with generic coefficients. %, i.e. their coefficients are variables.
 An important limitation in using \cite[Theorem 7]{RS} is a quite large number, $n(m + 1)$, of additional variables (i.e., coefficients of $L$ and $N$) needed to determine the polynomial $P_{f,L,N}$. This results in an exponential extension of the time needed for calculations. Theorem \ref{maintheorem1} provides a formula for the Łojasiewicz exponent without using additional variables. However, we need to repeat the calculations $\ell_N\binom{\ell_L}{n}$ times. From the point of view of the formula's effectiveness and algorithmic implementation, this is a significant simplification  and  reduces the time of computer calculations considerably.
%\end{rem}

\section{Preliminaries}

\subsection{Basic notations}
By $f\colon(\cc^m,0)\to(\cc^n,0)$ we denote a mapping defined in a neighbourhood of $0 \in \cc^m$ with values in $\cc^n$ such that $f(0)=0$.
A holomorphic mapping $f\colon(\cc^m,0)\to(\cc^n,0)$ is called \emph{finite at $0$} if $0$ is an isolated point of $f^{-1}(0)$.

We denote by $\cc[x]$  the ring of complex polynomials in  $x=(x_1,\ldots,x_m)$. The degree of a polynomial $f\in \cc[x]$ is denoted by $\deg f$. We put $\deg 0=-\infty$. Let $f=(f_1,\ldots,f_r):\cc^m\to\cc^r$ be a polynomial mapping, i.e., $f_j\in\cc[x]$ for $1\le j\le r$. By the degree of $f$, denoted  by $\deg f$, we mean $\max\{\deg f_j:1\le j\le r\}$.

For any set $\II\subset \cc[x]$,  we denote by $V(\II)\subset \cc^m$ the set of common zeros of all polynomials $f\in\II$;
%% ??
any set of the form $V(\II)$ is called \emph{algebraic}. We put $V(\emptyset)=\cc^m$. By the Hilbert basis theorem $V(\II)=V(\JJ)=V(f_1,\ldots,f_r)$ for some polynomials $f_j\in\cc[x]$, $1\le j\le r$, belonging to the ideal $\JJ\subset \cc[x]$ generated by $\II$. If the polynomials $f_j$, $1\le j\le r$, are homogeneous, we call  the set $V(f_1,\ldots,f_r)$ an \emph{algebraic cone}.

 We denote by $\dim V$ the dimension of an algebraic (or locally analytic) set $V\subset \cc^n$. We set $\dim\emptyset=-1$. Let $\dim_a V$  denote the dimension of $V$ at $a\in\cc^n$, i.e., the dimension of the germ of the set $V$ at $a$. The set $V$ is called \emph{pure dimensional} if all irreducible components of $V$ have the same dimension.

\subsection{Tangent cone}
Let $C_0(V)$ be the \emph{tangent cone} to a set $V\subset \cc^m$ at $0$ in the sense of Whitney \cite{Whitney}, i.e., the set of vectors $w\in\cc^m$ for which there are sequences $p_\nu\in V$ with $p_\nu\to 0$ and $\alpha_\nu\in\cc$ such that $\alpha_\nu p_\nu\to w$ as $\nu\to\infty$.  We have the following fact (cf. \cite[Theorem 8.10 and Lemma 8.11]{Whitney}).

\begin{fact}\label{factpropertiestangcone}
Let $V\subset \cc^m$ be an analytic set in a neighbourhood of the origin
 with $0\in V$. Then $C_0(V)$ is an algebraic cone and $\dim C_0(V)=\dim_0V$. If $0$ is a simple point of $V$ then $C_0(V)$ is the tangent space $T_0(V)$ to $V$ at $0$.
\end{fact}

\subsection{Degree of algebraic sets}
 The \emph{degree} of a pure dimensional algebraic set $V\subset \cc^m$, denoted by $\deg V$, 
is defined as the  number of common points of $V$ and a generic affine subspace $X\subset \cc^m$ of dimension $m-\dim V$;
and the \emph{degree of  $V$ at a point} $a\in V$, denoted by $\deg_a V$,  is the multiplicity $i_a(V,X)$ of proper intersection of $V$ with  $X$ as above such that $a\in X$  (see \cite{Draper}).  Obviously, $\deg_a V\le \deg V$.

Let $V=V_1\cup\cdots\cup V_s$ be the decomposition of an algebraic set $V\subset \cc^m$   into irreducible components. The number
%By $\delta(V)$ we denote the \emph{total degree} of the algebraic set $V\subset \cc^m$, i.e., the number
$$
\delta (V):=\deg V_1+\cdots+\deg V_s
$$
is called the \emph{total degree} of  $V$ (see \cite{Lo3}).%, where $\deg V_j$ is the degree of $V_j$.

We have the following useful properties of the total degree (see \cite{Lo3}).

\begin{fact}\label{deltainters}\label{Fact1}\label{Fact2}
If $V,W\subset \cc^m$ are algebraic sets, then
\begin{align}
&\delta(V\cap W)\le \delta(V)\delta(W),\label{deltaintersection}\\
&\delta(V)\le \deg f_1\cdots \deg f_r,\quad\hbox{where }V=V(f_1,\ldots,f_r),\label{deltaproduct}\\
&\delta(\overline{L(V)})\le \delta(V)\quad\hbox{for any linear mapping }L:\cc^m\to\cc^k.\label{deltalinear}
%&\delta(C_0(V))\le \delta(V).\label{deltaoftangentcone}
\end{align}
\end{fact}

\begin{fact}\label{factdegreeoftangcone}
If $V\subset \cc^m$ is an algebraic set, then $\delta (C_0(V))\le \delta (V)$. If $f:\cc^n\to\cc^m$ is a polynomial mapping of degree $d$,  finite at $0$, then $\delta\big(C_0\big(\overline{f(\cc^n)}\big)\big)\le d^{n}$.%, where $W=\overline{f(\cc^n)}$.
\end{fact}

\begin{proof} The first  assertion immediately follows from \cite[Theorem 6.4]{Draper}. We will prove the second assertion. Let $W=f(\cc^n)$,  $k=\dim \overline{W}$ and  $D=\deg \overline{W}$. Then $k\le n$ and by the definition of the degree of an algebraic set and the fact that  $\dim\overline{\overline{W}\setminus W}<k$, for a generic linear mapping $L:\cc^m\to\cc^k$ and a generic $z\in\cc^k$, $\#L^{-1}(z)\cap W=D$,
  %% ??
  where $\#$ stands for cardinality.
  Hence $D\le \delta((L\circ f)^{-1}(z))$. Since for $y\in\cc^m$ such that $L(y)=z$ we have $(L\circ f)^{-1}(z)=(L\circ (f-y))^{-1}(0)$, \eqref{deltaproduct} gives $D\le d^k\le d^n$. Consequently, the first assertion implies the second.
\end{proof}

%\subsection{Auxiliary lemma}
We denote by
$\mathbb{L}(m,r)$ the set of linear mappings $\mathbb{C}^m\rightarrow \mathbb{C}^r$, where $m,r \in \mathbb{N}$.  We set $\mathbb{L}(m,r)= \{0\}$ if  $r=0$.

We will need the following lemma (cf. \cite[Lemma 3.20]{KOSS} and  \cite[Lemma 1.1]{S0}).

\begin{lem}\label{Fact3} Let  $f=(f_1,\ldots,f_r):\cc^n\to\cc^r$ be a polynomial mapping such that $\dim V(f)=q$. Let $W$ be the union of the irreducible components  of dimension $q$ of the set $V(f)$. Then for a generic linear mapping $L=(L_1,\ldots,L_{n-q})\in\mathbb{L}(r,n-q)$ the set $V_1=V(L_1\circ f,\ldots,L_{n-q}\circ f)$ has dimension $q$, $W$ is the union of some irreducible components of $V_1$, and
$$
\deg W\le \delta(V_1)\le \deg (L_1\circ f)\cdots\deg (L_{n-q}\circ f).
$$
\end{lem}

\begin{proof} Without loss of generality we may assume that $q\ge 0$.
Let $Y=\overline{f(\cc^n)}\subset \cc^r$. By Chevalley's theorem (see \cite[p. 395]{Lo3}) the set $Y$ is algebraic and obviously irreducible. Moreover,  \cite[Theorem  3.13]{M}  easily implies that $r\ge \dim_0 Y=\dim Y\ge n-q$, and there exists a Zariski open and dense subset $U$ of $Y$ such that $\dim f^{-1}(y)\le q$ for $y\in U$ and obviously $\dim (Y\setminus U)\le \dim Y-1$.

Let $k=\dim C_0(Y)$. By Fact \ref{factpropertiestangcone}, $k=\dim Y$ and so $r\ge k\ge n-q$.

If $k<r$, by definition of the degree of an algebraic set, for a generic linear mapping $L:\cc^r\to\cc^{k}$ the set $Y\cap L^{-1}(0)$ is finite,  $(Y\setminus U)\cap L^{-1}(0)\subset \{0\}$ and obviously $0\in Y\cap L^{-1}(0)$. So,  $V(L\circ f)$ has dimension $q$ and contains $W$. Consequently, $W$ is the union of some irreducible components of $V(L\circ f)$. Hence, it suffices to consider the case $k=r$, and then obviously $Y=\cc^r$ and $r\le n$.

Let
$$
Z=\{(x,L)\in \cc^n\times \mathbb{L}(r,n-q):L\circ f(x)=0\},
$$
and let
$$
\pi:Z\ni (x,L)\mapsto L\in \mathbb{L}(r,n-q).
$$
Obviously $\dim \mathbb{L}(r,n-q)=r(n-q)$. Since $Y=\cc^r$ and $r\le n$, we have $\overline{\pi(Z)}=\mathbb{L}(r,n-q)$, $\dim Z=r(n-q)+q$ and for a generic $L\in \mathbb{L}(r,n-q)$ the set $\pi^{-1}(L)$ has dimension $q$. This gives the first  assertion. The second follows immediately from the first  and Fact \ref{deltainters}.
\end{proof}

From Lemma \ref{Fact3} we immediately obtain

\begin{wn}\label{factdeltaofacone}
If $V=V(f_1,\ldots,f_r)\subset \cc^n$ is an algebraic cone of pure dimension $q$, where $f_1,\ldots,f_r:\cc^n\to\cc$ are nonzero homogeneous polynomials with $\deg f_j\le d$ for $1\le j\le r$, then
$$
\delta(V)\le d^{n-q}.
$$
\end{wn}

\subsection{Projective algebraic sets}
Let $\PP^m$ be the $m$-dimensional complex projective space.
Let $V\subset \PP^m$ be an algebraic set, i.e., the set of common zeros of some system $S\subset \cc[x_0,\ldots,x_m]$ of homogeneous polynomials. Let  $\II\subset \cc[x_0,\ldots,x_m]$ be the radical of the ideal generated by $S$. Then $\II$ is the ideal of all polynomials vanishing on $V$ and
$$
V^\cc:=\{(x_0,\ldots,x_m)\in \cc^{m+1}:[x_0:\ldots:x_m]\in V\}\cup\{0\}=V(\II),
$$
where $V(\II)=\{x\in \cc^{m+1}:f(x)=0\hbox{ for any }f\in \II\}$, is an algebraic cone. We have $\dim V=\dim V^\cc-1$. The set $V$ is called \emph{pure dimensional} if all irreducible components of $V$ have the same dimension.

Since any hyperplane $H\subset \cc^{m+1}$ of dimension $k$ can be written in the form $V(L_1,\ldots,L_{m-k},L_{m-k+1}-a)$ for some $L_1,\ldots,L_{m-k+1}\in \mathbb{L}(m,1)$ and $a\in\cc$, we can define the \emph{degree} of  $V$ by $\deg V=\deg V^\cc$ (if $V$ is of pure dimension), and its \emph{total degree} by $\delta(V)=\delta(V^\cc)$.

\section{Auxiliary results}

\subsection{Main lemma}

A system of functions $N_1,\ldots,N_s\in\mathbb{L}(m,1)$, $s\ge m$, will be called \emph{independent} if for any sequence $1\le i_1<\cdots<i_m\le s$ the system  $N_{i_1},\ldots,N_{i_m}$ is linearly independent over $\mathbb{C}$.

\begin{rem}\label{exa1}
Take a system $N_1,\ldots,N_s\in \mathbb{L}(m,1)$ with $s\ge n$, and let $N_i(x)=a_{i,1}x_1+\cdots+a_{i,m}x_m$. Then the system $N_1,\ldots,N_s$ is independent if and only if $\prod_{1\le i_1<\cdots<i_m\le s}\det[a_{i_j,l}]_{1\le j,l\le m}\ne 0$. So, the set of independent systems is a dense subset of $\mathbb{L}(m,s)$ with algebraic complement.
\end{rem}

\begin{exa}\label{exaexample}
For any injective sequence $a_i\in\cc$, $1\le i\le s$, $s\ge m$, the system
$$
N_i(x_1,\ldots,x_m)=x_1+a_ix_2+\cdots+a_i^{m-1}x_m,\quad 1\le i\le s,
$$
is independent. Indeed, this follows from the properties of the Vandermonde determinant.
 \end{exa}

A crucial role in the discussions below will be played by the following lemma.

\begin{lem}\label{arange0} Let $m,q,d$ be positive integers with $m\ge 2$, $m\ge q$, let $\ell=d(m-q)+q$, and let $N_j\in\mathbb{L}(m,1)$, $1\le j\le \ell$, be a system of independent linear functions. Then for any algebraic cone $C_0\subset \cc^m$ with $\dim C_0\le q$ and $\delta(C_0)\le d$, there exist $1\le i_1<\cdots<i_q\le \ell$ such that
\begin{equation*}%\label{arrangeN}
C_0\cap V(N_{i_1},\ldots,N_{i_q})=\{0\}.% \quad\hbox{for some }i\in\{1,\ldots,s\}.
\end{equation*}
\end{lem}

\begin{proof} The proof is by induction on $q\in\{1,\ldots,m\}$. Take any algebraic cone $C_0\subset \cc^m$ with $\dim C_0=1$ and $\delta(C_0)\le d$. Then there are $w_1,\ldots, w_d\in\cc^m\setminus\{0\}$ such that  $C_0=w_1\cc\cup\cdots\cup w_d\cc$. Suppose that
for any $i\in\{1,\ldots,\ell\}$ %, where $\ell=d(m-1)+1$,
there exists $\varphi(i)\in\{1,\ldots,d\}$ such that $N_i(w_{\varphi(i)})=0$. Since $w_{\varphi(i)}\ne 0$,  the choice of $N_1,\ldots,N_\ell$ yields $\#\varphi^{-1}(j)\le m-1$, so% for any $j\in \{1,\ldots,d\}$. So,
$$
d(m-1)+1=\ell= \#\varphi^{-1}(\{1,\ldots,d\}) = \#\varphi^{-1}(1)+\cdots+\#\varphi^{-1}(d)\le d(m-1),
$$
which is impossible.
% (by $\# A$ we denote the cardinality of a set $A$).
This gives the assertion for $q=1$.

Assume that the assertion holds for $1\le q-1<m$. Take any algebraic cone $C_0$ with $\dim C_0\le q$ and $\delta(C_0)\le d$. If $q=m$ then the assertion is obvious. Assume that $q<m$. Let $C_0=V_1\cup\cdots\cup V_s$ be the decomposition into irreducible components. Since $\delta(C_0)\le d$, we have  $s\le d$. Observe that there exists $i\in\{1,\ldots,\ell\}$, %where $\ell=d(m-q)+q$,
 such that $\dim C_0\cap V(N_i)<q$. Indeed, suppose that for any $i\in\{1,\ldots,\ell\}$ there exists $\psi(i)\in \{1,\ldots,s\}$ such that $V_{\psi(i)}\subset V(N_i)$ and $\dim V_{\psi(i)}=q$. Analogously to the above, the choice of $N_1,\ldots,N_\ell$ yields $\#\psi^{-1}(j)\le m-q$, so
$$
d(m-q)+q=\ell=\psi^{-1}(\{1,\ldots,s\})%=\#\psi^{-1}(1)+\cdots+\#\psi^{-1}(s)
\le s(m-q)\le d(m-q),
$$
which is impossible and gives the announced observation.

By the above observation we can take a function $N_i$ with $\dim C_1\le q-1$ and $\delta(C_1)\le d$, where $C_1=C_0\cap V(N_i)$. After a linear change of coordinates, we can assume that $N_i(x_1,\ldots,x_m)=x_m$. % and then $M_j=N_j|_{\cc^{m-1}\times\{0\}}$ for $j\ne i$ is an independent function system. Moreover
Then the family $N_j|_{\cc^{m-1}\times\{0\}}$, $j\ne i$, contains $d(m-q)+q-1=d((m-1)-(q-1))+q-1$ independent linear functions in $x_1,\ldots,x_{m-1}$. So, the induction hypothesis gives the assertion for $q$ and the proof  is complete. 
\end{proof}

\begin{rem}\label{lemthebest}
The number $\ell=d(m-q)+q$ in Lemma \ref{arange0} is optimal. Indeed, take an independent system of functions $N_j\in\mathbb{L}(m,1)$, $1\le j\le \ell$. Let $$
V_i=V(N_{(i-1)(m-q)+1},\ldots,N_{(i-1)(m-q)+m-q}),\quad i=1,\ldots,d.
$$
Then $C_0=V_1\cup\cdots\cup V_d\subset \cc^m$  is an algebraic cone, $\dim(C_0)=q$, $\delta(C_0)=d$, and
\begin{align}
&C_0\cap V(N_{i_1},\ldots,N_{i_q})\ne \{0\}\quad \hbox{for }1\le i_1<\cdots<i_q\le \ell\hbox{ with }i_1\le d(m-q),\nonumber\\
&C_0\cap V(N_{d(m-q)+1},\ldots,N_{d(m-q)+q})= \{0\}.\nonumber
\end{align}
\end{rem}

\subsection{Corollaries from Lemma \ref{arange0}}

From Lemma \ref{arange0} we immediately obtain (see \cite[Theorem 6.3]{Draper})

\begin{wn}%\label{Bertinimultiple}
Let $d,m,q$ be positive integers such that $m\ge q$, let
$
\ell=  d(m-q)+q%d^{m-q}((m-q)(d-1)+q-1)+m(q-1)-1,
$,
 and let $N_j\in\mathbb{L}(m(q-1),1)$, $1\le j\le \ell$, be a system of independent linear functions.
Then  for any  algebraic set $V=V(f_1,\ldots,f_r)\subset \cc^m$ of pure dimension $q$, where $f_j\in\cc[x]$, $\deg f_j\le d$ for $1\le j\le r$, we have
$$
\deg _0V=\min_{1\le i_1<\cdots<i_{q}\le \ell}i_0(V,V(N_{i_1},\ldots,N_{i_q})).%=\min_{1\le i_1<\cdots<i_{q}\le \ell}i(C_0(V),V(N_{i_1},\ldots,N_{i_q})).
$$
\end{wn}

\begin{wn}\label{remproperintersection}
Let $C_0\subset \cc^m$ be an algebraic cone of pure dimension $q$ and degree $d$.

{\rm(a)} If  $N_1,\ldots,N_q\in\mathbb{L}(m,1)$ is a system of linear functions such that $C_0\cap V(N_1,\ldots,N_q)=\{0\}$, then for any
%%??
$1\le i_1<\cdots <i_s\le q$ the intersection $C_0\cap V(N_{i_1},\ldots,N_{i_s})$ is proper.

{\rm(b)} For any system of independent linear functions $N_j\in\mathbb{L}(m,1)$, $1\le j\le \ell$, $\ell=d(m-q)+q$, where $d=\deg C_0$, there exist $1\le j_1<\cdots<j_q\le \ell$ such that for any $1\le i_1<\cdots,i_s\le q$ the intersection $C_0\cap V(N_{j_{i_1}},\ldots,N_{j_{i_s}})$ is proper.
\end{wn}

\begin{proof}
Let $C_1$ be an irreducible component of $C_0$. Take any $1\le i_1<\cdots<i_s\le q$ and let $\sigma:\{1,\ldots,q\}\to\{1,\ldots,q\}$ be a bijection such that $\sigma(j)=i_j$ for $1\le j\le s$. By \cite[Theorem 3.13]{M} for any $1\le i \le q$ we have
\begin{equation*}
\dim (C_1\cap V(N_{\sigma(1)},\ldots,N_{\sigma(i-1)}))-1\le\dim (C_1\cap V(N_{\sigma(1)},\ldots,N_{\sigma(i)})).
%\\  \le \dim (C_1\cap V(N_{\sigma(1)},\ldots,N_{\sigma(i-1)})).
\end{equation*}
Since $\dim C_1=q$ and $\dim (C_1\cap V(N_{\sigma(1)},\ldots,N_{\sigma(q)}))=0$, for $1\le i\le q$ we have
$$
\dim (C_1\cap V(N_{\sigma(1)},\ldots,N_{\sigma(i-1)}))-1 =\dim (C_1\cap V(N_{\sigma(1)},\ldots,N_{\sigma(i)})),
$$
and consequently $\dim (C_1\cap V(N_{\sigma(1)},\ldots,N_{\sigma(s)}))=q-s$, which gives (a). From (a) and Lemma \ref{arange0} we immediately obtain (b). 
\end{proof}

From Lemma \ref{arange0} for $q=m-1$ we immediately obtain

\begin{wn}\label{cornotin}
Let $m,d$ be positive integers with $m\ge 2$, let $\ell=d+m-1$, and let $N_j\in\mathbb{L}(m,1)$, $1\le j\le \ell$, be a system of independent linear functions. Then  for any algebraic cone $C_0\subset \cc^m$ with $\dim C_0<m$ and $\delta(C_0)\le d$, there exist $1\le i_1<\cdots<i_{m-1}\le \ell$  such  that  $V(N_{i_1},\ldots,N_{i_{m-1}})\cap C_0\subset\{0\}$. Moreover, the sets $V(N_{i_1},\ldots,N_{i_{m-1}})$, $1\le i_1<\cdots<i_{m-1}\le \ell$, are one-dimensional linear subspaces of $\cc^m$.
\end{wn}

From Corollary \ref{cornotin} and Example \ref{exaexample} we immediately obtain

\begin{wn}\label{cornodinacone}
Let $N_j(x_1,\ldots,x_m)=x_1+a_jx_2+\cdots+a_j^{m-1}x_m$, where $a_j\in \cc\setminus\{0\}$ are pairwise different numbers, $1\le j\le \ell$,
$\ell=d+m-1$, $d>0$.
Then for any $1\le j_1<\cdots<j_{m-1}\le \ell$ there exists a unique point
\begin{equation*}\label{eqsystemequations}
F(a_{j_1},\ldots,a_{j_{m-1}})\in V(N_{j_1},\ldots,N_{j_{m-1}})
\end{equation*}
%has a unique solution $F(a_{j_1},\ldots,a_{j_{m-1}})\in \cc^m$
with  first coordinate  $1$. Moreover, for any algebraic cone $C\subset \cc^m$ with $\dim C<m$ and $\delta(C)\le d$ there exist $1\le {j_1}<\cdots <{j_{m-1}}\le \ell$ such that $F(a_{j_1},\ldots,a_{j_{m-1}})\notin C$.
\end{wn}

\begin{rem}\label{explicitformula} Let $n=m-1>0$. By using the Lagrange interpolation formula we easily find that the function
$$
F=(F_0,\ldots,F_{n}):\cc^{n}\ni(a_1,\ldots,a_{n})\mapsto F(a_1,\ldots,a_{n})\in\cc^{n+1}
$$
in Corollary \ref{cornodinacone} is defined by $F_0(a_1,\ldots,a_{n})=1$ and
$$
F_k(a_1,\ldots,a_n)=(-1)^{n-k+1}\sum_{j=1}^{n}\Bigg(a_j\underset{{\begin{smallmatrix}1\le i\le m\\ i\ne j \end{smallmatrix}}}\prod(a_j-a_i)\Bigg)^{-1} %\sum_{\begin{smallmatrix}1\le i\le m\\ i\ne j \end{smallmatrix}}
\!\!\!\!\!\sum_{{\begin{smallmatrix} 1\le i_1<\cdots< i_{n-k}\le n\\ i_s\ne j  \end{smallmatrix}}}a_{i_1}\cdots a_{i_{n-k}}.
$$
for $k=1,\ldots, n$, where $a_i\ne 0$ and $a_i\ne a_j$ for $i\ne j$.
%
%The points $x_{j_1,\ldots,j_{n-1}}\in \cc^m$ from Corollary \ref{cornodinacone} are of the form
%$x_{j_1,\ldots,j_{n-1}}=(y_1,\ldots, y_n)$, where $y_1=1$ and
%$$
%y_{j+1}=\sum_{i=1}^{m-1}\frac{-1}{a_{j_i}}\sum_{1\le s_1<\cdots<s_{n-k}\le m-1,\;s_i\ne j}
%$$
\end{rem}

We have the following generalization of Lemma \ref{arange0}.

\begin{wn}\label{lemeverysystem}
Let $k,m,q,d$ be positive integers with $m\ge 2$, $m\ge q$ and $k\ge q$. Let
\begin{equation}\label{eqformellk}
\ell_k=k+d(m-1)\binom{k}{q}-d(q-1)%\sum_{j=q}^{k-1}\binom{j}{q-1},
\end{equation}
 and let $N_j\in\mathbb{L}(m,1)$, $1\le j\le \ell_k$ be a system of independent linear functions. Then for any algebraic cone $C_0\subset \cc^m$ with  $\dim C_0\le q$ and $\delta(C_0)\le d$, there exist $1\le j_1<\cdots<j_k\le \ell_k$ such that for any $1\le i_1<\cdots<i_q\le k$,
\begin{equation}\label{arrangeN}
C_0\cap V(N_{j_{i_1}},\ldots,N_{j_{i_q}})=\{0\}.% \quad\hbox{for some }i\in\{1,\ldots,s\}.
\end{equation}
\end{wn}

\begin{proof}
The proof is by induction on $k\ge q$. If $k=q$ then the assertion follows from Lemma \ref{arange0}. Assume that the assertion holds for $k-1\ge q$. Take an algebraic cone $C_0\subset \cc^m$ with $\dim C_0\le q$ and $\delta(C_0)\le d$. By the induction hypothesis there are $1\le j_1<\cdots<j_{k-1}\le \ell_{k-1}$ such that for any $1\le i_1<\cdots<i_q\le k-1$ the equality \eqref{arrangeN} holds. Let
\begin{equation}\label{eqc1}
C_1=\bigcup_{1\le i_1<\cdots<i_{q-1}\le k-1}[C_0\cap V(N_{j_{i_1}}, \ldots,N_{j_{i_{q-1}}})].
\end{equation}
By \eqref{arrangeN} we see that $\dim C_1\le 1$ and by \eqref{deltaintersection} in Fact \ref{deltainters}, $\delta (C_1)\le \binom{k-1}{q-1}d$. Since $\ell_k-\ell_{k-1}=\binom{k-1}{q-1}d(m-1)+1$, by Lemma \ref{arange0} there exists $j_k\in\{\ell_{k-1}+1,\ldots,\ell_k\}$ such that $C_1\cap V(N_{j_k})=\{0\}$. So,
%%?? niezbyt formalnie jasne
for $j_1,\ldots,j_k$ we easily obtain the assertion %for $k$ 
and the proof  is complete. 
\end{proof}

\begin{rem}\label{lemthebestellk}
The number $\ell_k$ %=d(m-q)+q$
 in Corollary \ref{lemeverysystem} is optimal. Indeed, by Remark \ref{lemthebest},  $\ell_q=\ell$ is optimal for $k=q$. Assume that  $\ell_{k-1}$ is optimal for $k-1\ge q$. Then by arguing as in the proof of Corollary \ref{lemeverysystem} we see that there exists an algebraic cone $C_0\subset \cc^n$ of pure dimension $q$ with $\delta(C_0)=d$ such that the cone $C_1$ defined by \eqref{eqc1} is of pure dimension $1$ and $\delta (C_1)= \binom{k-1}{q-1}d$. So, by Remark \ref{lemthebest}, we need an additional $\binom{k-1}{q-1}d(m-1)+1$ linear forms to obtain the assertion of Corollary \ref{lemeverysystem} for $k$. Consequently, $\ell_k=\ell_{k-1}+\binom{k-1}{q-1}d(m-1)+1$ and $\ell_k$ is of the form \eqref{eqformellk}, which completes the proof. %{\bf Sprawdzic ??????}
\end{rem}

\begin{wn}\label{lempolynomials1}
Let $V\subset \cc^m$ be an algebraic cone of pure dimension $q$ and degree $d$. There are homogeneous polynomials $f_1,\ldots,f_{m-q}\in\cc[x_1,\ldots,x_m]$ such that $\deg f_j\le d$ for $1\le j\le m-q$, the algebraic cone $W=V(f_1,\ldots,f_{m-q})$ has pure dimension $q$ and degree not exceeding  $d^{m-q}$, and $V$ is the union of some irreducible components of $W$. Moreover, the matrix
$$
J(f_1,\ldots,f_{m-q}):=\left[\frac{\partial f_{j}}{\partial x_k}\right]_{\begin{smallmatrix}1\le j\le m-q\\
1\le k\le m
\end{smallmatrix}}
$$
has rank $m-q$ on a Zariski open and dense subset of $W$. %In particular, the set $V^*$ of singular points of $V$ is an algebraic cone which is contained in an algebraic cone $Z\subset V$ such that $\dim Z\le q-1$ and $\delta(Z)\le d(d-1)^{n-q}$.
\end{wn}

\begin{proof}
If $q=m$ then the assertion is trivial. Assume that $q<m$.
By the Sadullaev theorem (see \cite[p. 389]{Lo3}) for $k=q+1$ the set $S_{m-k}(V)$ of Sadullaev's spaces of dimension $m-k$ for $V$ (i.e., linear spaces $Y\subset \cc^m$ of dimension $m-k$ such that $
V\subset \{x+y:x\in X,\;y\in Y,\;|y|\le C(1+|x|)\}$ for some $C>0$, where $X$ is
%%?? chyba jest wiele linear complements, wiec "a"
a linear complement of $Y$), is an open and dense subset of the Grassmann space $G_{m-k}(\cc^m)$ (of $m-k$-dimensional linear subspaces of $\cc^m$) with complement being a proper algebraic set. These are exactly the  spaces $Y\in G_{m-k}(\cc^m)$ for which $V\cap Y=\{0\}$.
 So, for any $Y\in S_{m-k}(V)$ and its linear complement $X$, the restriction to $V$ of the projection $\pi_Y:\cc^n=X+Y\ni x+y\mapsto x\in X$ is a proper mapping.  By  Chevalley's theorem (see \cite[p. 395]{Lo3}), $\pi_Y(V)\subset X$ is a proper algebraic set and obviously it is an algebraic cone. Moreover, from \cite[Theorem  3.13]{M} we easily deduce that $\pi_Y(V)$ has pure dimension $q=\dim X-1$, and by Fact \ref{deltainters}, $\delta(\pi_Y(V))\le d$. So, there exists a homogeneous polynomial without multiple factors $f_Y$ of degree not exceeding $d$ such that $\pi_Y(V)+Y=V(f_Y)$. Consequently, $f_Y(x+y)=f_Y(x)$ for $x\in X$ and $y\in Y$ and $f_Y$ vanishes on $V$.

Since $V\cap Y=\{0\}$ for $Y\in S_{m-k}(V)$, by Corollary \ref{lemeverysystem} (for $k=n$) and Remark \ref{exa1} (or Corollary \ref{cornodinacone})
%Using the above, by easy induction with respect to the dimmension of minimal (with respect to inclussion) linear space containing the algebraic cone $V$
 we easily see that there exists a system of coordinates $x_1,\ldots,x_m$ of $\cc^m$ such that for
\begin{equation*}
\begin{split}
Y_{j}&=\{(x_1,\ldots,x_m)\in\cc^m:x_s=0\hbox{ for }j\le s\le j+q\},\\
X_{j}&=\{(x_1,\ldots,x_m)\in\cc^n:x_s=0\hbox{ for }1\le s\le j-1\hbox{ and for }j+q+1\le s\le m\},
\end{split}
\end{equation*}
$1\le j\le m-q$, we have $Y_j\in S_{m-k}(V)$, $X_j$ is
%%??
a linear complement of $Y_j$ in $\cc^m$. Moreover, one can assume that
$$
Z_j=\{(x_1,\ldots,x_m)\in X_j:x_s=0\hbox{ for }s\ne j\}
$$
is a Sadullaev space of $\pi_{Y_j}(V)\subset X_j$ for $1\le j\le m-q$. Consequently, the polynomials $f_{Y_j}$ are of the forms
\begin{multline*}
f_{Y_j}(x_j,\ldots,x_{j+q})=f_{0,Y_j}(x_{j+1},\ldots,x_{j+q})x_j^{N_j}+f_{1,Y_j}(x_{j+1},\ldots,x_{j+q})x_j^{N_j-1}\\
+\cdots+f_{N_j,Y_j}(x_{j+1},\ldots,x_{j+q})
\end{multline*}
with $N_j>0$ and $f_{0,Y_j}(x_{j+1},\ldots,x_{j+q})\ne 0$. Moreover, for any $(x_{j+1},\ldots,x_{j+q})\in \cc^q$
%%?? niejasne - każdy wielomian ma skonczenie wiele zer - ?
the polynomial $f_{Y_j}$ has a finite number of zeros $x_j$. So, the projection onto $\{0\}\times \cc^q$ of any irreducible component of  the set
\begin{equation*}\label{eqintersectpiY11}
W=\bigcap_{1\le j\le m-q}(\pi_{Y_j}(V)+Y_j)
\end{equation*}
is proper and $W$ has pure dimension $q$.  Since $V$ has pure dimension $q$ and $V\subset W$,  we infer that $V$ is the union of some irreducible components of $W$. So, by Fact \ref{deltainters} we obtain $\deg V\le \deg W\le d^{m-q}$.

 Since $f_{Y_j}$ has no multiple factors, we have % one can assume that $\deg _{x_{j} }f_{Y_J}\ge 1$ and consequently,
 ${\partial f_{Y_{j}}}/{\partial x_j}\ne 0$ on any irreducible component of $W$ for any $1\le j\le m-q$.
By the definition of $f_{Y_j}$ we have ${\partial f_{Y_{j}}}/{\partial x_k}=0$ in $\cc^n$ for $1\le k\le j-1$ and $1\le j\le m-q$. So, %we obtain the desired claim. Consequently, the set of singular points of $V$ is contained in the zeroset $Z_M$ of any $(m-q)\times(m-q)$ minor $M$ of the Jacobian matrix $J(f_{Y_1},\ldots,f_{Y_{m-q}})$. Moreover,
the minor
$$
M_0=\det\left[\frac{\partial f_{Y_j}}{\partial x_k}\right]_{\begin{smallmatrix}1\le j\le m-q\\
1\le k\le m-q
\end{smallmatrix}}
$$
 is nonzero on any irreducible component of $W$. This gives the assertion.
\end{proof}

\section{Bertini's theorem}

\subsection{Bertini's weak theorem}
For any $a\in\cc$, denote by $N_a$ the linear function $x_1+ax_2+\cdots+a^{m-1}x_m$. We have 

\begin{wn}[Bertini's weak theorem]\label{properintersect1}
Let $C\subset \cc^m$ be an algebraic cone of pure dimension $q\ge 1$. Then the set $A\subset \cc$ of  points $a\in \cc$ such that $C\cap V(N_a)$ is an improper intersection is finite. Moreover, $\# A\le d(m-q)$, where $\delta(C)\le d$.% is the degree of $C$.
\end{wn}

\begin{proof}
Let $k=d(m-q)$. Suppose to the contrary that $\#A>k$. Then there are $a_1,\ldots,a_{k+1}\in A$ such that $a_i\ne a_j$ for $1\le i,j\le k+1$. Consequently, $\dim [C\cap V(N_{a_j})]=q$ for $1\le j\le k+1$. Then for any $a_{k+2},\ldots, a_{k+q}\in\cc$ and any $1\le i_1<\cdots<i_q\le d(m-q)+q$ we have $\dim [C\cap V(N_{a_1},\ldots,N_{a_q})]>0$. This contradicts Lemma \ref{arange0} and Example \ref{exaexample}.% and ends the proof.
\end{proof}

The properness of the  intersection $C\cap V(N_a)$ in Corollary \ref{properintersect1} cannot be replaced by transversality, as shown by the following example:

\begin{exa}\label{counterexample}
Let 
$$
C=\{(x,y,z)\in\cc^3:y^2-4xz=0\}.
$$
Obviously $C\setminus\{0\}$ is a smooth cone of dimension $2$. Let $N_a(x,y,z)=x+ay+a^2 z$, $a\in\cc$. Then $C\cap V(N_a)\ne \emptyset$ for any $a\in\cc$. Moreover, if $(0,0,0)\ne (x_0,y_0,z_0)\in C\cap V(N_a)$ then $z_0\ne 0$, $a=\frac{-y_0}{2z_0}$ and $x_0=\frac{y_0^2}{4z_0}$. So,
$$
N_a(x,y,z)=x+\frac{-y_0}{2z_0}y+\frac{y_0^2}{4z_0^2}z=\frac{-1}{4z_0}\left(-4z_0x+2y_0y-4x_0 z \right)
$$
and $T_{(x_0,y_0,z_0)}C=V(N_a)$. This shows that the intersection $C\cap V(N_a)$ is  not transversal at $(x_0,y_0,z_0)$.
\end{exa}

\subsection{Effective Bertini theorem}

In this section we will prove some effective version of Bertini's theorem. Let us start with a lemma.

\begin{lem}\label{lendegreeconetangents}
Let $V\subset \cc^m$ be an irreducible algebraic cone of dimension $q>0$ and degree $d>0$. Assume that $V\setminus \{0\}$ is smooth. Let
$$
C=\{N\in \LL(m,1):\exists_{x\in V\setminus\{0\}}\;N(x)=0\;\land\; T_x(V)\subset V(N)\}.
$$
Then $\overline{C}$ is an irreducible algebraic cone in $\cc^m$ with $\dim \overline{C}<m$ and
\begin{equation}\label{eqestdegC}
\deg \overline{C}\le 2d^{m-q}[(m-q)(d-1)+1]^{q-1}.
\end{equation}
\end{lem}

\begin{proof}
The inequality $\dim \overline{C}<m$ follows immediately from the fact that for any $x\in V\setminus\{0\}$ the set of $N\in \LL(m,1)$ such that $N(x)=0$ and $T_x(V)\subset V(N)$ is a linear space of dimension $m-q$ and that $T_{\lambda x}(V)=T_x(V)$ for $\lambda\in\cc\setminus\{0\}$. Since the set $V\setminus\{0\}$ is smooth and connected, the tangent bundle
$$
W=\{(x,y)\in (V\setminus\{0\})\times \cc^n: y\in T_x(V)\}
$$
of $V\setminus\{0\}$ is a smooth and connected manifold, so the set $\{(x,N)\in (V\setminus\{0\})\times \LL(m,1):T_x(V)\subset V(N)\}$ is also a smooth and connected manifold. In particular, the closure of this set is an irreducible algebraic set. Consequently, $\overline{C}$ is irreducible, as the projection onto $\LL(m,1)$ of the above set.

We will prove \eqref{eqestdegC} by induction on $m$. For $m=1$ the assertion is trivial. Assume that it holds for $m-1$. %We will prove it  for $m$.
Without loss of generality, we can assume that  $V\not\subset \{(x_1,\ldots,x_m)\in\cc^m:x_1=0\}$.
By Corollary \ref{lempolynomials1} there are nonzero homogeneous polynomials $f_1,\ldots,f_{m-q}\in\cc[x]$  such that $\deg f_j\le d$ for $1\le j\le m-q$, the algebraic cone $V_1=V(f_1,\ldots,f_{m-q})$ has pure dimension $q$ and  $\deg V_1\le d^{m-q}$, and $V$ is an irreducible component of $V_1$. Moreover, the matrix
$$
J(f_1,\ldots,f_{m-q}):=\left[\frac{\partial f_{j}}{\partial x_k}\right]_{\begin{smallmatrix}1\le j\le m-q\\
1\le k\le m
\end{smallmatrix}}
$$
has rank $m-q$ on a Zariski open and dense subset of $V_1$. Let $b_{j,k}=\frac{\partial f_{j}}{\partial x_k}$ for $1\le j\le m-q$, $1\le k\le m$ and $b_{m-q+1,k}=a_k$ for $1\le k\le m$, where $a_k\in \cc$, $N(x_1,\ldots,x_m)=a_1 x_1+\cdots+a_mx_m$, and let
$B=[b_{j,k}]_{\begin{smallmatrix}1\le j\le m-q+1\\
1\le k\le m \end{smallmatrix}}$.  Let
\begin{multline*}
X=\{(x,N)\in \cc^m\times\LL(m,1):x\ne 0,\; f_1(x)=\cdots=f_{m-q}(x)=0,\;N(x)=0,\\
\rk B(x,N)\le m-q\}.
\end{multline*}
Then there exists an irreducible component $X_1$ of $\overline{X}$ such that $\overline{C}$ is equal to the closure of the projection of $X_1$ onto $\LL(m,1)$. So, by Fact \ref{deltainters},
\begin{equation}\label{eqestdegC1}
\deg \overline{C}\le \deg X_1\le \delta(\overline{X}).
\end{equation}

Since $x_1$ does not vanish on $V$, we have $X_1\not\subset\{0\}\times\cc^{m-1}$. Moreover,
%and let $Y\subset \LL(m,1)$ be the closure of the projection of $X$ onto $\LL(m,1)$. Then $C$ is an irreducible component of $Y$.
%Since $x_1$ not vanishes on $V$,
by the Euler formula $\frac{\partial f_j}{\partial x_1}x_1+\cdots+\frac{\partial f_j}{\partial x_m}x_m=f_j\deg f_j$, the set $X_1$ is an irreducible component of an algebraic set described by the equations $f_1(x)=\cdots=f_{m-q}(x)=0$, $N(x)=0$ and by
%%??
$q-1$ minors of $B$ of size $(m-q+1)\times(m-q+1)$ and  of degree $(m-q)(d-1)+1$. So, using Fact \ref{deltainters} and \eqref{eqestdegC1} we obtain the assertion.
\end{proof}

Set $ab^T=a_1b_1+\cdots+a_mb_m$ for $a=(a_1,\ldots,a_m), b=(b_1,\ldots,b_m)\in\cc^m$.

\begin{tw}[Bertini]\label{Bertini}
Let $d,m,q$ be positive integers such that $m\ge q$, let
$$
\ell= 2 d^{m-q}[(m-q)(d-1)+1]^{q-1}+m-1,
$$
and let $N_j\in\mathbb{L}(m,1)$, $1\le j\le \ell$, be a system of independent linear functions. Then
\[
E_{j_1,\ldots,j_{m-1}}=\{x\in\cc^m:\exists_{a\in\cc^m\setminus\{0\}}\;N_{j_1}(a)=\cdots =N_{j_{m-1}}(a)=0,\; ax^T=0\},
\]
for $1\le j_1<\cdots<j_{m-1}\le \ell$ is a system of hyperplanes such that for any irreducible  algebraic cone $V\subset \cc^m$ with   $\dim V = q$ and $\deg V\le d$ such that $V\setminus \{0\}$ is smooth, there are $1\le j_1<\cdots<j_{m-1}\le \ell$ such that
the intersection $X=V\cap E_{j_1,\ldots,j_{m-1}}$ is transversal at any point $x\in X\setminus \{0\}$ and the set  $X\setminus\{0\}$ is smooth. If moreover $q\ge 3$ then  $X$ is irreducible.
\end{tw}

% Then for any $1\le j_1<\cdots<j_{m-1}\le \ell$ the set
%\begin{equation}\label{eqformEj}
%E_{j_1,\ldots,j_{m-1}}=\{a\in \cc^m:N_{j_1}(a)=\cdots=N_{j_{m-1}}(a)=0\}
%\end{equation}
%is a linear subspace of $\cc^m$ with dimension $1$. Moreover,

\begin{proof}
The first  assertion follows immediately from Corollary \ref{cornotin} and Lemma \ref{lendegreeconetangents}. The second  assertion is immediate by arguing as in    \cite[proof of Theorem II.8.18]{H}. % the book by Robin Hartshorne, \emph{Algebraic Geometry}, Springer Science+Business Media, Inc. 1977.
\end{proof}

\begin{rem}\label{remspecialcase}
The assertion of Theorem \ref{Bertini} holds for the system of linear functions
$$
N_j(x_0,\ldots,x_n)=x_0+a_jx_1+\cdots +a_j^{n}x_m,
$$
where $a_j\in\cc$, $1\le j\le \ell$, are pairwise different numbers (see Corollary \ref{cornodinacone} and Remark \ref{explicitformula}).
\end{rem}

From Theorem \ref{Bertini} we immediately obtain its version for projective varieties.

\begin{tw}[Bertini]\label{BertiniProj}
Let $d,m,q$ be positive integers such that $m\ge q$, let
$$
\ell=  d^{m-q}[(m-q)(d-1)+1]^{q}+m,
$$
 and let $N_j\in\mathbb{L}(m+1,1)$, $1\le j\le \ell$, be a system of independent linear functions.
 Then
\[
F_{j_1,\ldots,j_{m-1}}=\{x\in\PP^m:\exists_{a\in\PP^m}\;N_{j_1}(a)=\cdots =N_{j_{m-1}}(a)=0,\; ax^T=0\}
\]
for $1\le j_1<\cdots<j_{m-1}\le \ell$ is a system of hyperplanes such that for any  irreducible smooth algebraic set $V\subset \PP^m$ with $\dim V = q$ and $\deg V\le d$, there are $1\le j_1<\cdots<j_{m}\le \ell$ such that the intersection $X=V\cap F_{j_1,\ldots,j_{m-1}}$ is transversal at any point $x\in X$ and the set  $X$ is smooth. If moreover $q\ge 2$ then  $X$ is irreducible.
\end{tw}

By a similar argument to the proof of Lemma \ref{lendegreeconetangents} we will obtain

\begin{lem}\label{lendegreeconetangentsmulti}
Let $V\subset \cc^m$ be an irreducible algebraic cone of dimension $q>0$ and degree $d>0$. Assume that $V\setminus \{0\}$ is smooth. Let %$1\le s<q$ and let
$$
C=\{N\in \LL(m,q-1):\exists_{x\in V\setminus\{0\}}\;N(x)=0\;\land\; \dim[T_x(V)\cap V(N)]>1\}.
$$
Then $\overline{C}$ is an algebraic cone in $\LL(m,q-1)$ with $\dim \overline{C}<m(q-1)$ and
\begin{equation}\label{eqestdegCmilti}
\delta(\overline{C})\le d^{m-q}[(m-q)(d-1)+q-1].
\end{equation}
\end{lem}

\begin{proof}
Let $\II\subset \cc[x_1,\ldots,x_m]$ be the ideal of $V$. For any $f_1,\ldots,f_{m-q}\in\II$ and $N=(N_1,\ldots,N_{q-1})\in\LL(m,q-1)$, where
\begin{equation}\label{eqformN}
N_j(x_1,\ldots,x_m)=a_{j,1} x_1+\cdots+a_{j,m}x_m,
\end{equation}
  $a_{j,k}\in \cc$ for  $1\le j\le q-1$ and $1\le k\le m$, we put
\begin{equation*}
\begin{split}
b_{j,k}&=\frac{\partial f_{j}}{\partial x_k}(x)\quad\hbox{for }1\le j\le m-q,\;\; 1\le k\le m,\\
b_{n-q+j,k}&=a_{j,k}\quad\hbox{for }1\le j\le q-1,\;\; 1\le k\le m.
\end{split}
\end{equation*}
Let
$$
B(x,f_1,\ldots,f_{m-q},N)=[b_{j,k}]_{\begin{smallmatrix}1\le j\le m-1\\
1\le k\le m \end{smallmatrix}}.
$$
Then
\begin{multline*}
C=\{N\in \LL(m,q-1):\exists_{x\in V\setminus\{0\}}\;N(x)=0\\
\land\;\forall_{f_1,\ldots,f_{m-q}\in\II} \land \rk B(x,f_1,\ldots,f_{m-q},N)< m-1\}.
\end{multline*}

%Let
%$$
%\Gamma= {\{(x,y)\in\cc^m\times \cc^m:x\in V\setminus\{0\},\;y\in T_x(V)\}}.
%$$
%Then $\Gamma$ is the tangent bundle of the set $V\setminus\{0\}$. So, it is a smooth and connected analytic set in $[\cc^m\setminus\{0\}]\times \cc^m$ of dimension $2q$. Consequently $\overline{\Gamma}$ is an irreducible algebraic cone of dimension $2q$. Let
%$$
%\Delta=\{(x,y,N)\in\Gamma\times \LL(m,q-1):N(x)=N(y)=0\}.
%$$
%$$%\begin{multline*}
%\Delta_1=\{(x,y,N)\in \Delta: \dim[ T_x(V)\cap V(N)]>1\}
%$$%\end{multline*}
%and let
%$$
%W=\{(x,N)\in \cc^m\times \LL(m,q-1):x\in V\setminus\{0\},\;N(x)=0,\; \dim[ T_x(V)\cap V(N)]>1\}.
%$$
%Then $\Delta$ is a smooth irreducible analytic set  in $[\cc^m\setminus\{0\}]\times \cc^m$ of dimension

Take any $x=(x_1,\ldots,x_m)\in V\setminus\{0\}$. After a linear change of coordinates, one can assume that $T_x(V)=\cc^q\times \{0\}\subset \cc^m$ and $x_1\ne 0$. So, for $N=(N_1,\ldots,N_{q-1})\in\LL(m,q-1)$ of the form \eqref{eqformN} we have $N(x)=0$ and $ \dim[ T_x(V)\cap V(N)]>1$ if and only if $N(x)=0$ and
$$
\det[a_{j,i}]_{\begin{smallmatrix}1\le j\le q-1\\
2\le i\le q
\end{smallmatrix}}=0,
$$
where we have used the Euler formula (by an analogous argument to the proof of Lemma \ref{lendegreeconetangents}). Consequently, the set of $N\in \LL(m,q-1)$ such that  $ \dim[ T_x(V)\cap V(N)]>1$ and $N(x)=0$ is an irreducible algebraic set of dimension $m(q-1)-q$. So, the fibers of the mapping
$$
W\ni (x,N)\mapsto x\in V\setminus\{0\}
 $$
are irreducible sets, where
$$
W=\{(x,N)\in \cc^m\times \LL(m,q-1):x\in V\setminus\{0\},\;N(x)=0,\; \dim[ T_x(V)\cap V(N)]>1\}.
$$
 Since $V\setminus\{0\}$ is a smooth set, we easily deduce that $\pi$ is a locally trivial fibration. Thus $\overline{W}$ is an irreducible algebraic set,  and hence so is $\overline{C}$  (as the projection of $\overline{W}$ onto $\LL(m,q-1)$).

 The inequality $\dim \overline{C}<m(q-1)$ follows immediately from the fact that for any $x\in V\setminus\{0\}$ the set of $N\in \LL(m,q-1)$ such that  $ \dim[ T_x(V)\cap V(N)]>1$ and $N(x)=0$ is an algebraic set of dimension $m(q-1)-q$,
%%??
 because $\dim V=q$ and  $T_{\lambda x}(V)\cap V(N)=T_x(V)\cap V(N)$ for $\lambda\in\cc\setminus\{0\}$.

We will prove \eqref{eqestdegCmilti} by induction on $m$. For $m=1$ the assertion is trivial. Assume that it holds for $m-1$. Without loss of generality, we can assume that  $V\not\subset \{(x_1,\ldots,x_m)\in\cc^m:x_1=0\}$.
By Corollary \ref{lempolynomials1} there are nonzero homogeneous polynomials $f_1,\ldots,f_{m-q}\in\II$  such that $\deg f_j\le d$ for $1\le j\le m-q$, the algebraic cone $V_1=V(f_1,\ldots,f_{m-q})$ has pure dimension $q$ and degree $\le d^{m-q}$, and $V$ is an irreducible component of $V_1$. Moreover, the matrix
$$
J(f_1,\ldots,f_{m-q}):=\left[\frac{\partial f_{j}}{\partial x_k}\right]_{\begin{smallmatrix}1\le j\le m-q\\
1\le k\le m
\end{smallmatrix}}
$$
has rank $m-q$ on a Zariski open and dense subset of $V_1$.
%Let
%$b_{j,k}=\frac{\partial f_{j}}{\partial x_k}(x)$ for $1\le j\le m-q$, $1\le k\le m$ and $b_{n-q+j,k}=a_{j,k}$ for $1\le j\le s$ and $1\le k\le m$, where $a_{j,k}\in \cc$, $N_j(x_1,\ldots,x_m)=a_{j,1} x_1+\cdots+a_{j,m}x_m$, and let
%$B(x,N_1,\ldots,N_s)=[b_{j,k}]_{\begin{smallmatrix}1\le j\le m-q+s\\
%1\le k\le m \end{smallmatrix}}$.
Let
\begin{multline*}
X=\{(x,N)\in \cc^m\times\LL(m,q-1):f_1(x)=\cdots=f_{m-q}(x)=0,\;N(x)=0,\\
\rk B(x,f_1,\ldots,f_{m-q},N)< m-1\}
\end{multline*}
and let $Y\subset \LL(m,1)$ be the closure of the projection of $X$ onto $\LL(m,s)$. Then $\overline{C}$ is an irreducible component of $Y$.

Since $x_1$ does not vanish on $V$, by the Euler formula % $\frac{\partial f_j}{\partial x_1}x_1+\cdots+\frac{\partial f_j}{\partial x_m}x_m=f_j\deg f_j$,
%%?? czy to jest jasne?
the set $X$ can be described by $q-1$ minors of $B$  of size $(m-1)\times(m-1)$ and of degree $(m-q)(d-1)+q-1$ and by the equations $f_1(x)=\cdots=f_{m-q}(x)=0$, $N(x)=0$. So, using Fact \ref{deltainters} we obtain the assertion.
\end{proof}

From Corollary  \ref{cornotin} and Lemma \ref{lendegreeconetangentsmulti} we immediately obtain

\begin{tw}[Bertini]\label{Bertinimultiple}
Let $d,m,q,s$ be positive integers such that $m\ge q$ and $q-1 \ge s$, let
$$
\ell=  d^{m-q}[(m-q)(d-1)+q-1]+m(q-1)-1,
$$
 and let $N_j\in\mathbb{L}(m(q-1),1)$, $1\le j\le \ell$, be a system of independent linear functions. Then
\begin{multline*}
K_{s,j_1,\ldots,j_{m(q-1)-1}}=\{x\in\cc^m:\exists_{a=(a_1,\ldots,a_{q-1})\in (\cc^m)^{q-1}\setminus\{0\}}\;N_{j_1}(a)=\cdots \\=N_{j_{m(q-1)-1}}(a)=0,\; a_1x^T=\cdots=a_{s}x^T=0\}\quad\hbox{for } 1\le j_1<\cdots<j_{m-1}\le \ell
\end{multline*}
is a system of linear subspaces of dimension $m-s$ such that for any irreducible  algebraic cone $V\subset \cc^m$ with  $\dim V = q$ and $\deg V\le d$ such that $V\setminus \{0\}$ is smooth, there are $1\le j_1<\cdots<j_{m(q-1)-1}\le \ell$ such that
%for any
%$$
%(a_{1,1},\ldots,a_{1,m},a_{2,1},\ldots,a_{2,m},\ldots,a_{q-1.1},\ldots,a_{q-1,m})\in V(N_{j_1},\ldots,N_{m(q-1)-1})\setminus\{0\},
%$$
% any $1\le s\le q$ and the linear mapping
%$$
%N_{s,j_1,\ldots,j_{m(q-1)-1}}(x_1,\ldots,x_m)=(a_{j,1}x_1+\cdots+a_{j,m}x_m:1\le j \le s)\in \LL(m,s),
%$$
 the intersection $X=V\cap V(K_{s,j_1,\ldots,j_{m(q-1)-1}})$ is transversal at any point $x\in X\setminus \{0\}$ and the set  $X\setminus\{0\}$ is smooth. If moreover $q \ge s+2$ then  $X$ is irreducible.
\end{tw}

\begin{rem}\label{rembeterest}
%It is worth noting that
The estimate of the number of linear functions $\ell$ in Theorem \ref{Bertinimultiple} is better than the one obtained by repeated use of the estimate of $\ell$ in Theorem \ref{Bertini}.
\end{rem}

\begin{rem}\label{remprojmultisection}
 It is easy to state the above theorem for sets in projective spaces.
\end{rem}

From Theorem \ref{Bertinimultiple} we immediately obtain (cf. Corollary \ref{Bertinimultiple})

\begin{wn}\label{Bertinimultipleextra}
Under the notations and assumptions of Theorem \ref{Bertinimultiple},
%Let $d,m,q$ be positive integers such that $m\ge q$, let
%$$
%\ell=  d^{m-q}[(m-q)(d-1)+q-1]+m(q-1)-1,
%$$
%and let
%$$
%X_{s,j_1,\ldots,j_{m(q-1)-1}}=V(N_{s,j_1,\ldots,j_{m(q-1)-1}})\quad\hbox{for }1\le j_1<\cdots<j_{m(q-1)-1}\le \ell
%$$
%for $1\le s\le q$.% and let $N_j\in\mathbb{L}(m(q-1),1)$, $1\le j\le \ell$ be a system of independent linear functions.
  for any  algebraic set $V=V(f_1,\ldots,f_r)\subset \cc^m$ of pure dimension $q$, where $f_j\in\cc[x]$, $\deg f_j\le d$ for $1\le j\le r$, we have
$$
\deg _0V=\max_{1\le j_1<\cdots<j_{m(q-1)-1}\le \ell}\deg_0 (V\cap K_{s,j_1,\ldots,j_{m(q-1)-1}})%=\min_{1\le i_1<\cdots<i_{q}\le \ell}i(C_0(V),V(N_{i_1},\ldots,N_{i_q})).
$$
for any $1\le s\le q$.
\end{wn}

\section{A formula for the multiplicity of a finite mapping at zero}

Let  $f=(f_1,\ldots,f_m):(\cc^n,0)\to(\cc^m,0)$ be a finite mapping. Then obviously $m\ge n$. By the \emph{multiplicity} of $f$ at $0$ we mean
the improper intersection multiplicity
$$
i(\graph f\cdot(\cc^n\times \{0\});(0,0))
$$
of the $\graph f$ and $\cc^n\times \{0\}$ at  $(0,0)\in
\cc^n\times \cc^m$ (see \cite{ATW}, \cite{Stoll}) and we denote it by $i_0(f)$. Note that $i_0(f)=\infty$ iff $f$ is not finite. If $m=n$ then $i_0(f)=\mu_0(f)$, where $\mu_0(f)$ denotes the covering multiplicity of
$f$ at $0$, or equivalently the codimension of the ideal $(f_1,\ldots,f_n)$ in the ring  $\OO$ of germs of holomorphic functions at $0\in\cc^n$. More precisely,
\begin{equation}\label{eqpropermultipl}
i_0(f)=\mu_0(f)=\dim_\cc\OO/(f_1,\ldots,f_n).
\end{equation}
In the case $m>n$, we have
$$
i_0(f)\ge \dim_\cc\OO/(f_1,\ldots,f_m).
$$

 %The set of  independent systems of functions $N_1,\ldots,N_s\in\mathbb{L}(n,1)$, $s\ge n$,  is a dense subset of $\mathbb{L}(m,s)$ with algebraic complement (see Remark \ref{exa1}).

We will prove the following effective formula for the multiplicity $i_0(f)$ for finite polynomial mappings $f=(f_1,\ldots,f_m):\cc^n\to\cc^m$.
%By the degree of $f$ we mean $\max\{\deg f_1,\ldots,\deg f_m\}$.

\begin{tw}\label{twmultiplicityformula}
Let $\ell=  d^{n}(m-n)+n$.
For any independent system
$$
L_j\in \mathbb{L}(m,1),\quad 1\le j\le \ell,%d^{n-1}(m-n)+n,
$$
and any  polynomial mapping $f\colon\cc^n\to\cc^m$ of degree $d$, finite at $0$,% i.e., $d=\max\{\deg f_1,\ldots,\deg f_m\}$,
$$
i_0(f)=\min_{1\le i_1<\cdots<i_n\le\ell}% d^{n-1}(m-n)+n}
\dim_\cc\OO/(L_{i_1}\circ f,\ldots,L_{i_n}\circ f).
$$
\end{tw}

%The proof of Theorem \ref{twmultiplicityformula} will be preceded by a proposition.
From \cite[Theorem 4.4]{ATW}, analogously to \cite[Theorem 1.1]{S} we get  

\begin{pr}[]\label{propSpodz}
Let $f:(\cc^n,0)\to(\cc^m,0)$ be a finite mapping and let $C_0$ be the tangent cone to the set $f(U)$ for a sufficiently small neighbourhood $U\subset \cc^n$ of the origin. Then for any linear mapping $L\in \mathbb{L}(m,n)$ we have $i_0(f)\le i_0(L\circ f)$, and  equality holds if and only if $V(L)\cap C_0=\{0\}$.
\end{pr}

\begin{proof}[Proof of Theorem \ref{twmultiplicityformula}]
Since $f$ is a polynomial mapping finite at $0$, we have $\dim \overline{f(\cc^n)}=n$ and by Fact \ref{factpropertiestangcone}, $\dim \big(C_0\big(\overline{f(\cc^n)}\big)\big)=n$. From the definition of $d$, Fact \ref{factdegreeoftangcone} gives $\delta\big(C_0\big(\overline{f(\cc^n)}\big)\big)\le d^{n}$. So, by Lemma \ref{arange0} there are $1\le i_1,\ldots,i_n\le \ell$ such that $V(L_{i_1},\ldots,L_{i_n})\cap C_0\big(\overline{f(\cc^n)}\big)=\{0\}$. Combining \eqref{eqpropermultipl} and Proposition~\ref{propSpodz} gives the assertion.
\end{proof}

By an analogous argument to the proof of Theorem \ref{twmultiplicityformula} we obtain

\begin{wn}\label{coreffectivemult}
Let $\ell=  d(m-n)+n$.
For any independent system
$$
L_j\in \mathbb{L}(m,1),\quad 1\le j\le \ell,%d^{n-1}(m-n)+n,
$$
and any  polynomial mapping $f\colon\cc^n\to\cc^m$  with $\deg_0\overline{f(\cc^n)}=d$, finite at $0$,% i.e., $d=\max\{\deg f_1,\ldots,\deg f_m\}$,
$$
i_0(f)=\min_{1\le i_1<\cdots<i_n\le\ell}% d^{n-1}(m-n)+n}
\dim_\cc\OO/(L_{i_1}\circ f,\ldots,L_{i_n}\circ f).
$$
\end{wn}

From Bezout's theorem and Theorem \ref{twmultiplicityformula} we obtain 

\begin{wn}\label{cordimension}
Let   $f\colon\cc^n\to\cc^m$ be a polynomial mapping of degree $d$ such that $f(0)=0$.
Let $m\ge n> q\ge 0$, $A=d^n(m-n)+n$, $B=d^n(n-q)+q$, and let
\begin{align}
&L^q_i\in \mathbb{L}(m+q,1),\quad 1\le i\le A,\nonumber\\ %\label{eqL}\\
&M^q_j\in\mathbb{L}(n,1),\quad 1\le j\le B,\nonumber% \label{eqN}
\end{align}
be systems of independent functions. Then
$$
\min_{\begin{smallmatrix}1\le i_1<\cdots<i_n\le A\\
1\le j_1<\cdots<j_q\le B
\end{smallmatrix}}\dim_\cc\OO/(L_{i_1}\circ (f,M_{j_1},\ldots,M_{j_q}),\ldots,L_{i_n}\circ (f,M_{j_1},\ldots,M_{j_q}))>d^n
$$
if and only if $\dim_0 V(f)\ge q+1$,
\end{wn}

\begin{rem}\label{remeffectivemult}
The above corollary gives an algorithm for computing $\dim_0 V(f)$ and deciding whether the mapping $f$ is finite.
\end{rem}

\section{Łojasiewicz exponent}\label{Lojexpsect}

%The local {\L}ojasiewicz exponent is an important tool in  singularity theory (for more detailed references see for instance \cite{RS}). This exponent is closely related to the order function  \cite{LT}, the polar quotients \cite{Melle}, \cite{Teissier}, the degree of $C^0$-sufficiency of jets determined by holomorphic functions \cite{Kui}, \cite{Kuo2}, \cite{BL}, \cite{ChangLu}, reduction of ideals \cite{Teissier}, optimization \cite{Schweighofer1}, etc. Therefore, finding formulas for and estimates of  $\cal{L}_0(f)$ is of interest. This problem was considered by many authors (see for instance \cite{BR}, \cite{CK}, \cite{Ko1},  \cite{Ko2}, \cite{Kuo3}--%,  \cite{LT},
%\cite{Lichtin}, \cite{P0}--\cite{P2}, \cite{Teissier}).

The \emph{local Łojasiewicz exponent} of a mapping $f\colon(\cc^n,0)\to(\cc^m,0)$, denoted by $\cal{L}_0(f)$, is defined to be the infimum of the set of all exponents $\nu\ge 1$ such that
\begin{equation*}
  |f(z)|\ge C|z|^\nu
\end{equation*}
for some constant $C>0$ in a neighbourhood of the origin in $\cc^n$  (see \cite{Teissier}).
In \cite{LT} it was proved that $\cal{L}_0(f)$ is a rational number and the infimum in the definition
is in fact a minimum, provided $f$ is finite at $0$.

%\section{Local Łojasiewicz exponent for overdetermined mappings}%\label{formula}

A. P{\l}oski \cite{P2} gave a formula for $\cal{L}_0(f)$ for a finite mapping $f$ in terms of the characteristic polynomial, provided $n=m$. In fact, the characteristic polynomial is a polynomial with holomorphic coefficients. In the case $n = m = 2$ such a formula was obtained by J. Chądzyński and T. Krasiński \cite{CK} in terms of resultants. In Proposition \ref{p0} below, we quote a version of  Płoski's  formula obtained by T.~Rodak and S.~Spodzieja \cite[Proposition 3]{RS}.
%This version not only refers to the characteristic polynomials.
%{\bf Zacytować Chądzyńskiego i Krasińskiego.}

Let $P(y,t)=\sum_{j=0}^\infty P_j(y)t^j\in \cc\{y,t\}$ be a  power
series regular of order $r+1$ in~$t$. Write
$$%\begin{equation}
  \Delta(P)=\min_{j=0}^r\frac{\ord P_j}{r+1-j}.
  %\label{e101}
$$%\end{equation}
%Let $P(y,t) \in \mathbb{C}\{y,t\}$ be a power series of the form%
%$$
% P(y,t)=\sum_{j=0}^{\infty }P_{j}(y)t^j,
%$$
%regular of order $r+1$ in $t$. Write
%$$
%\Delta (P)= \min _{j=0}^r \frac {\ord P_j}{r+1-j}.
%$$
Let $f\colon(\cc^n,0)\to (\cc^n,0)$ be finite at $0$. Then for every $k=1,\ldots,n$ there exists a power series $P_k(y,t) \in \mathbb{C}\{y,t\}$ such that for some arbitrarily small neighbourhoods $U_0$ and $W_0$ of $0 \in \mathbb{C}^n$ and $0 \in \mathbb{C}^{n+1}$ respectively, we have
\begin{equation}\label{maintrouble}
\{(y,t) \in W_0 : P_k(y,t)=0 \}=\{(f(z),z_k)\in\mathbb{C}^n\times\mathbb{C}:z\in U_0\}.
\end{equation}
%is equal to the image of $U_0$ under the mapping $(z_1,\ldots,z_n)=z \mapsto (f(z),z_k)$.

\begin{pr}\label{p0}
Under the above assumptions and notation,
\begin{equation}\label{Ploskiformula}\tag{P}
{\cal{L}}_{0}(f)= \max _{k=1}^n \frac{1}{\Delta (P_k)}.
\end{equation}
\end{pr}

The purpose of this section is to provide an effective version of the above proposition for polynomial mappings $f:(\cc^n,0)\to(\cc^m,0)$, $m\ge n$, finite at $0$ (see Theorem \ref{maintheorem1} below). Let us start with some supporting facts.

From  \cite[Theorem 2.1 and its proof]{S}  we have  the following proposition.

\begin{pr}[]\label{propSpodz2}
Let $f:(\cc^n,0)\to(\cc^m,0)$ be a finite mapping and let $C_0$ be the tangent cone to the set $f(U)$ for a sufficiently small neighbourhood $U\subset \cc^n$ of the origin. Then for any linear mapping $L\in \mathbb{L}(m,n)$ we have ${\cal{L}}_0(f)\le {\cal{L}}_0(L\circ f)$, and  equality holds if and only if $V(L)\cap C_0=\{0\}$.
\end{pr}

From Proposition \ref{propSpodz2}, Facts \ref{factpropertiestangcone}, \ref{factdegreeoftangcone} and Lemma \ref{arange0} we immediately obtain

\begin{wn}\label{reduct1}
Let $m\ge n$ be positive integers, let $\ell=d^n(m-n)+n$ and let $L_j\in\LL(m,1)$, $1\le j\le \ell$, be a system of independent linear functions. Then for any polynomial mapping $f:(\cc^n,0)\to(\cc^m,0)$ finite at $0$ with $\deg f\le d$ we have
$$
{\cal{L}}_0(f)=\min_{1\le i_1<\cdots<i_n\le \ell}{\cal{L}}_0((L_{i_1},\ldots,L_{i_n})\circ f).
$$
\end{wn}

We will need the following two propositions, proved by A.~Płoski  \cite{P1}, {\rm{\cite[Proposition 1.3, Theorem 3.5]{P2}}} for $m=n$ and generalized by S. Spodzieja {\rm{\cite[Corollary 1.3, Proposition 3.1]{S}}} to $m>n$.
% they was proved in \cite{P1}, {\rm{\cite[Proposition 1.3, Theorem 3.5]{P2}}} and in the case $m>n$ it was proved in {\rm{\cite[Corollary 1.3, Proposition 3.1]{S}}}.

\begin{pr}\label{p1}%[\rm\cite{P1}]\label{louo}
Let $f\colon(\cc^n,0)\to (\cc^m,0)$ be a holomorphic mapping finite at~$0$. If $g:(\mathbb{C}^n,0)\rightarrow (\mathbb{C}^m,0)$ is a holomorphic mapping such that $\ord(f-g)>{\cal{L}}_{0}(f)$ then $g$ is finite at $0$ and
 % \begin{equation*}
    ${\cal{L}}_{0}(g)={\cal{L}}_{0}(f)$.
 % \end{equation*}
\end{pr}

\begin{pr}\label{p2}%[\rm\cite{P1}]\label{wykladnikStopien}
Let $f\colon\cc^n\to\cc^m$ be a polynomial mapping finite at $0$ with $\deg f\le d$. Then
%  \begin{equation*}
    ${\cal{L}}_{0}(f)\le d^n$.
  %\end{equation*}
\end{pr}

We will consider the Łojasiewicz exponent of polynomial mappings.
Let $f=(f_1,\ldots,f_m)\colon\cc^n\to\cc^m$, $m\ge n$, be a polynomial mapping such that $f(0)=0$. Let $d$ be the degree of $f$.
The basic difficulty in obtaining effective formulas for the Łojasiewicz exponent is to determine the characteristic polynomials $ P_1, \ldots, P_n $ with holomorphic coefficients. To bypass this difficulty and get the usual polynomials, we will reduce the problem to finding the exponent for proper mappings. For this purpose we define a mapping $H_{L}\colon\mathbb{C}^n\rightarrow \mathbb{C}^n$ by
\begin{equation}\label{0}
H_{f,L}(z)=L\circ f(z)+(z_1^{{d^n}+1},\ldots,z_n^{{d^n}+1}),\quad z=(z_1,\ldots,z_n)\in\cc^n,
\end{equation}
where $d\ge \deg f$ and $L \in \mathbb{L}(m,n)$. Obviously $H_L$ is a proper mapping.

From Corollary \ref{reduct1} and Propositions \ref{p1}, \ref{p2} we obtain 

\begin{wn}\label{reduct11}
Let $m\ge n$ be positive integers, let $\ell=d^n(m-n)+n$, and let $L_j\in\LL(m,1)$, $1\le j\le \ell$, be a system of independent linear functions. Then for any polynomial mapping $f:(\cc^n,0)\to(\cc^m,0)$ finite at $0$ with $\deg f\le d$ we have
$$
{\cal{L}}_0(f)=\min_{1\le i_1<\cdots<i_n\le \ell}{\cal{L}}_0(H_{f,(L_{i_1},\ldots,L_{i_n})}).
$$
\end{wn}

Set
\begin{equation*}
\M(m,n)= \mathbb{L}(m,n)\times \mathbb{L}(n,1)
\end{equation*}
and let $\Phi_{f,L,N}:\cc^n\to\cc^{n+1}$, where  $(L,N)\in \M(m,n)$,
be  given by
$$%\begin{equation}\label{1}
\Phi_{f,L,N}(z)=(H_{f,L}(z),N(z)).
$$%\end{equation}
The mapping $\Phi_{f,L,N}$ is proper and consequently $\Phi_{f,L,N} (\cc^n)\subset \cc^{n+1}$
is an algebraic set of pure dimension $n$.
Hence, there exists a % an irreducible
polynomial $P_{f,L,N}\in \mathbb{C}[y,t]$,
where  $y=(y_1,\ldots,y_n)$ and $y_1,\ldots,y_n,t$ are independent variables,
of the form
%\begin{equation}\label{2}
$$
P_{f,L,N}(,y,t)=\sum ^{r_{f,L,N}}_{j=0}P_{f,L,N,j}(y)t^j
$$
%\end{equation}
such that
\begin{equation}\label{eqPhi1}
\Phi_{f,L,N} (\mathbb{C}^n)=V(P_{f,L,N}).
\end{equation}
Since $H_{f,L,N}$ is a proper mapping, we may assume that $P_{f,L,N,r_{f,L,N}} \not=0$. Moreover, we may assume that $P_{f,L,N}$ is irreducible.

To use the Płoski formula \eqref{Ploskiformula}, we have to guarantee the fulfillment of assumptions \eqref{maintrouble}. This can be done by selecting
%%??
sufficiently many independent linear functions of $L\in\LL(m,1)$ and $N\in\LL(n,1)$. Namely, let
\[
\ell_N=n+[(d^n+1)^n-1]n(n-1), \qquad \ell_L=d^n(m-n)+n.
\]
Take any families of independent linear functions
$$
N_i\in \mathbb{L}(n,1),\quad 1\le i\le \ell_N,\qquad %M_k\in \mathbb{L}(n,1),\; 1\le k\le \ell_M,
L_s\in \mathbb{L}(m,1),\quad 1\le s\le \ell_L.
$$
Let
\[%\begin{multline*}
\mathbb{I}_{L}=\{{\bf s}=(s_1,\ldots,s_n)\in \nn^n:  1\le s_1<\cdots<s_n\le \ell_L\}.
\]%\end{multline*}
Obviously $\#\mathbb{I}_{L}=\binom{\ell_L}{n}$.
Set
\[
L_{\bf s}=(L_{s_1},\ldots,L_{s_n})\in \mathbb{L}(m,n)\quad\hbox{for }{\bf s}\in\mathbb{I}_L.
\]
%Let $f:\cc^n\to\cc^m$ be a polynomial mapping of degree $d$ such that $f(0)=0$.
For any $1\le i\le \ell_N$ and  ${\bf s}\in \mathbb{I}_{L}$ we define a mapping $\Phi_{f,(i,{\bf s})}:\cc^n\to\cc^n\times \cc$ by %the following formula
$$
\Phi_{f,(i,{\bf s})}(z)=\Phi_{f,L_{\bf s},N_i},
$$
i.e., $\Phi_{f,(i,{\bf{s}})}=\left(H_{f,{\bf s}}(z),N_i(z)\right)$, where $H_{f,{\bf s}}(z)=L_{\bf s}(f(z)(z))+(z_1^{d^n+1},\ldots,z_n^{d^n+1})$.
%Let
%$P_{f,(i,{\bf s})}\in \cc[y,t]$ be of the form
%$$
%P_{f,(i,{\bf s})}=P_{f,L_{{\bf s}},N_i}.%\sum_{j=0}^{p_{(i,{\bf k},{\bf s})}}P_{(i,{\bf k},{\bf s}),j}(y)t^j
%$$

The main result of this section is the following theorem.

\begin{tw}\label{maintheorem1} For any polynomial mapping $f:%=(f_1,\ldots,f_m):
(\mathbb{C}^n,0)\to(\mathbb{C}^m,0)$ finite at $0$ of degree $d$ the {\L}ojasiewicz exponent ${\cal{L}}_0(f)$ is given by
$$
{\cal{L}}_0(f)=\min_{{\bf s}\in \mathbb{I}_L}\max_{1\le i\le\ell_N}\frac{1}{\Delta (P_{f,L_{\bf s},N_i})}.
$$
\end{tw}

We will precede the proof of the above theorem with a remark and an example.

\begin{rem}\label{remtomainthm1}
Since $P_{f,L,N}$ vanishes exactly on the image of $\Phi_{f,L,N}$, it can be effectively computed, for instance  by using Gr\"obner bases (see \cite{GP}). One can also compute it as a multipolynomial resultant of the coordinates of $\Phi_{f,L,N}$ (see \cite{Gelfand}).
An important point in this construction is properness of the mappings $H_{f,L,N}$ and $ \Phi_{f,L,N} $, which we owe to the component $(z_1^{{d^n}+1},\ldots,z_n^{{d^n}+1})$.  Therefore, the polynomials $ P_{f,L,N} $ in Theorem \ref{maintheorem1} are regular in $t$ and we can compute the numbers $\Delta(P_{f,L,N})$.
%%??
If we omit this component, with fixed $ L, M $ and $N=z_k$ the polynomials $P_{f,L,N}$ may not satisfy \eqref{maintrouble} (see Example \ref{p1new} below). %We overcome this difficulty by considering generic coefficients of the mappings.
\end{rem}

%In Section \ref{overdet} we give a generalization of the above theorem in the case of overdetermined mappings (see Corollary \ref{p4gener}).
\begin{exa}\label{p1new}
%Let $f=(f_1,\ldots,f_n):(\mathbb{C}^n,0)\to(\mathbb{C}^n,0)$ be a polynomial mapping of degree $d$. Take the polynomial mapping $H:\mathbb{C}^n\to\mathbb{C}^n$ defined by
%\begin{equation}\label{eqHeq}
%H:\mathbb{C}^n\ni (z_1,\ldots,z_n)=z \mapsto (f_1(z)+z_1^{d^n+1},\ldots,f_n(z)+ z_n^{d^n+1})\in\mathbb{C}^n.
%\end{equation}
%It is easy to see that  $H$ is proper, because it has no zeros at infinity. Therefore, the polynomials $ P $ in Theorem \ref{maintheorem1} are regular in $t$. Properness of $H$ is essential  to Theorem \ref{maintheorem1}, as illustrated by the following example.
The polynomial mapping $f=(f_1,f_2):\mathbb{C}^2\to\mathbb{C}^2$ defined by
$$
f_1(z_1,z_2)=z_1(1-z_1^2-z_2^2),\quad f_2(z_1,z_2)=z_2(1-z_1^2-z_2^2)
$$
is finite at $(0,0)\in\mathbb{C}^2$ and it is not proper. For the  irreducible polynomial
$$
P(y_1,y_2,t)=t^3(y_1^2+y_2^2)-ty_1^2+y_1^3
$$
we have $P(f(z_1,z_2),z_1)=0$ for $(z_1,z_2)\in\mathbb{C}^2$. Moreover, the polynomial $P$ is not  regular in $t$ and we cannot  compute the number $\Delta(P)$.
\end{exa}

\begin{proof}[Proof of Theorem \ref{maintheorem1}]

By an analogous argument to the proof of \cite[Lemma 1]{RS} we obtain 

\begin{lem}\label{RSformula}
If $P,Q,R\in \cc\{y,t\}$ are series such that
$$
P(y,t)=\sum_{j=0}^\infty P_j(y)t^j,\quad Q(y,t)=\sum_{j=0}^\infty Q_j(y)t^j,
$$
and $Q=PR$ and for some $r\ge 0$ we have $\ord P_j,\,\ord Q_j>0$, $j=0,\ldots,r$, then
\begin{equation}\label{eqineqRS}
\min_{0\le j\le r}\frac{\ord P_j}{r+1-j}\le \min_{0\le j\le r}\frac{\ord Q_j}{r+1-j}.
\end{equation}
If additionally $\ord R=0$, then  equality holds in \eqref{eqineqRS}.
\end{lem}

Fix ${\bf s}=(s_1,\ldots,s_n)$ with $1\le s_1<\cdots<s_n\le \ell_L$ and let $L_{{\bf s}}=(L_{s_1},\ldots,L_{s_n})\in \LL(m,n)$.

\begin{lem}\label{lemgenericN}
There are $1\le i_1<\cdots<i_n\le \ell_N$ such that
\begin{equation}\label{genericN}
V(H_{f,L_{\bf s}})\cap V(N_{i_j})=\{0\}\quad\hbox{for }1\le j\le n.
\end{equation}
\end{lem}

\begin{proof} Since  $H_{L,M}$ has no zeros at infinity, it is proper. Consequently, by Bezout's theorem it has at most $(d^n+1)^n-1$ zeros in $\cc^n\setminus\{0\}$. Thus by Corollary \ref{lemeverysystem} we obtain \eqref{genericN}.
\end{proof}

\begin{lem}\label{lemequalgerms}
If $N \in \mathbb{L}(n,1)$ satisfies \eqref{genericN}, then for any sufficiently small $\epsilon >0$ the germs of the sets
$$
A_\epsilon =\{(H_{f,L_{\bf s}}(z),N(z)):|z|<\epsilon \}\;\;  \hbox{ and }\;\;  B=\{(y,t):P_{f,L_{\bf s},N}(y,t)=0\}
$$
 at $0 \in \mathbb{C}^{n+1}$ are equal.
\end{lem}

\begin{proof}
It suffices to prove that for any $\epsilon >0$ there exists a neighbourhood $U_0$ of $0 \in \mathbb{C}^{n+1}$ such that $A_\epsilon \cap U_0=B\cap U_0$. Indeed, suppose there exists $\epsilon >0$ such that for any neighbourhood $U_0$ of zero we have $A_\epsilon \cap U_0\not=B\cap U_0$. Obviously $A_\epsilon \subset B$, hence $A_\epsilon \cap U_0\subset B\cap U_0$. So, the inclusion $A_\epsilon \cap U_0 \supset B\cap U_0$ fails, i.e., there exists a sequence $(y_\nu ,t_\nu )\in B\cap U_0$ such that $(y_\nu ,t_\nu )\rightarrow (0,0)$ and $(y_\nu ,t_\nu )\not\in A_\epsilon \cap U_0$.
By \eqref{eqPhi1} there exists a sequence $z_\nu \in \mathbb{C}^n$ such that $y_\nu =H_{L_{\bf s}}(z_\nu )$ and $t_\nu =N(z_\nu )$. Thus
$$
(H_{f,L_{\bf s}}(z_\nu ),N(z_\nu ))\rightarrow (0,0)
$$
as $\nu \rightarrow \infty $. The sequence $(z_\nu )$ is bounded (by properness of $H_{f,L_{\bf s}}$); we can assume that  $z_\nu \rightarrow z_0$ as $\nu \rightarrow \infty $. Then
$$
(H_{f,L_{\bf s}}(z_0 ),N(z_0 ))=(0,0),
$$
so, by \eqref{genericN}, $z_0=0$ and $z_\nu \rightarrow 0$ as $\nu \rightarrow \infty $.
On the other hand, $(y_\nu ,t_\nu )\not\in A_\epsilon \cap U_0$ implies $|z_\nu |\ge \epsilon $,  a contradiction.
\end{proof}

By Lemmas \ref{lemgenericN} and \ref{lemequalgerms}, for any fixed ${\bf s}$  we can find $1\le i_1,\ldots,i_n\le \ell_N$ such that $P_{f,L_{\bf s},N_{i_j}}$ satisfy the assumption \eqref{maintrouble} of Proposition \ref{p0}. Note that if  $P_{f,L_{\bf s},N_i}$ does not satisfy \eqref{maintrouble}, the characteristic polynomial for $\Phi_{f,L_{\bf s},N_i}$ is a divisor of $P_{f,L_{\bf s},N_i}$.  Hence, in view of Lemma \ref{RSformula}, we obtain
$$
{\cal{L}}_0(H_{f,L_{\bf s}})=\max_{1\le i\le \ell_N}\frac{1}{\Delta (P_{f,L_{\bf s},N_i})}.
$$
This together with Corollary \ref{reduct11} gives the assertion of Theorem \ref{maintheorem1}.
%Having regard to Lemma \ref{RSformula} and Corollary \ref{reduct11}, it suffices to show, that for any fixed $(L_{s_1},\ldots,L_{s_n})\in \LL(m,n)$, $1\le s_1<\cdots,s_n\le \ell_L$ there are $N_{i_1},\ldots,N_{i_n}\in\LL(n,1)$ such that
\end{proof}

Theorem \ref{maintheorem1} simplifies in the case of proper  polynomial mappings finite at $0$. By an analogous argument to that in the proof of Theorem \ref{maintheorem1} we obtain

\begin{wn}\label{cormaintheorem} Let $d,n,m$, $m\ge n$, be positive integers, let
\[
\ell_N=(d^n-1)n(n-1)+n,\quad \ell_L=d^n(m-n)+n,
\]
and let $N_i\in\mathbb{L}(n,1)$, $1\le i\le \ell_N$, $L_j\in \mathbb{L}(m,1)$, $1\le j\le \ell_L$, be sequences of independent functions.
Then for any proper polynomial mapping $f:(\mathbb{C}^n,0)\to(\mathbb{C}^n,0)$ finite at $0$ of degree $d$ there exist polynomials
\[
P_{f,(L_{s_1},\ldots,L_{s_n}),N_i} \in \mathbb{C}[y,t],\quad 1\le i\le \ell_N,\quad 1\le s_1<\cdots<s_n\le \ell_L,
\]
 regular in $t$,  such that  the set $V(P_{f,(L_{s_1},\ldots,L_{s_n}),N_i})$  is equal to the image of $\mathbb{C}^n$ under the mapping $z \mapsto ((L_{s_1},\ldots,L_{s_n})\circ f(z),N_i(z))$ and
$$
\mathcal{L}_0(f)=\min_{1\le s_1<\cdots<s_n\le\ell_L}\max_{1\le i\le\ell_N}\frac{1}{\Delta (P_{f,(L_{s_1},\ldots,L_{s_n}),N_i})}.
$$
%  for $k=1,\ldots, s$, and formula \eqref{6}.
%$$
%{\cal{L}}_{0}(f)= \min_{1\le i_1<\cdots,i_n\le s}\max _{k=1}^n \frac{1}{\Delta (P_{i_k})}.
%$$
\end{wn}

Theorem \ref{maintheorem1} and Corollary \ref{cormaintheorem} give effective algorithms for computing the multiplicity and the  local {\L}ojasiewicz exponent of a polynomial mapping $\mathbb{C}^n\to\mathbb{C}^n$ finite at $0$. Moreover, they allow one to decide whether a polynomial mapping $(\mathbb{C}^n,0)\to(\mathbb{C}^n,0)$ is finite at $0$. Namely, by calculating the Łojasiewicz exponents of mappings $(f, M):\cc^n\to\cc^{m+q}$ for $M\in\LL(n,q)$ instead of $f$, we obtain

\begin{pr}\label{pr5newdim}
Let $d,n,m$, $m\ge n$, be positive integers, let $0\le q\le n$, and let
\[
%\ell_N=n+[(d^n+1)^n-1]n(n-1), \quad
\ell_L=d^n(m-n)+n, \quad\ell_M=d^n(n-q)+q.
\]
Let %$N_i\in\mathbb{L}(n,1)$, $1\le i\le \ell_N$,
$L_j\in \mathbb{L}(m,1)$, $1\le j\le \ell_L$, $M_k\in\LL(n,1)$, $1\le k\le \ell_M$, be sequences of independent functions. Let $f:(\cc^n,0)\to\cc^m,0)$ be a polynomial mapping of degree $d$. Set
$$
H_{f,L_{\bf s},M_{\bf k}}(z)=L_{\bf s}\circ(f(z),M_{\bf k}(z))+\left(z_1^{d^n+1},\ldots,z_n^{d^n+1}\right)
$$
for $L_{\bf s}=(L_{s_1},\ldots,L_{s_n})$, ${\bf s}=(s_1,\ldots,s_n)\in \mathbb{I}_L$ and $M_{\bf k}=(M_{k_1},\ldots,M_{k_q})$, ${\bf k}=(k_1,\ldots,k_q)$, $1\le k_1<\cdots<k_q\le\ell_M$.

\indent 1. If $\dim_0 V(f)\ge q+1$ then
$$
\min_{{\bf s},\,{\bf k}}\mathcal{L}_0(H_{f,L_{\bf s},M_{\bf k}})\ge d^n+1.
$$

\indent 2. If $\dim_0 V(f)\le q$ then
$$
\min_{{\bf s},\,{\bf k}} \mathcal{L}_0(H_{f,L_{\bf s},M_{\bf k}})\le d^n.
$$
\end{pr}

%\bibliography{bibliografia}{}
%\bibliographystyle{plain}

\end{document}